\documentclass[a4paper,11pt]{amsart}
\usepackage{a4wide}

\usepackage[american]{babel}
\usepackage[utf8]{inputenc}
\usepackage[T1]{fontenc}
\usepackage{lmodern}
\usepackage{geometry}
\usepackage{cite}
\usepackage{amsmath}
\usepackage{amssymb}
\usepackage{hyperref}
\usepackage{amsthm}

\def\IR{\mathbb{R}}
\def\IC{\mathbb{C}}

\def\IN{\mathbb{N}}

\def\E{\mathcal{E}}
\def\H{\mathcal{H}}

\def\a{\mathfrak{a}}
\def\b{\mathfrak{b}}

\def\lra{\longrightarrow}

\DeclareMathOperator{\supp}{supp}

\DeclareMathOperator{\sgn}{sgn}
\renewcommand{\Re}{\operatorname{Re}}

\def\loc{\mathrm{loc}}
\newcommand{\vol}{\mathrm{vol}}
\newcommand{\End}{\mathrm{End}}
\newcommand{\capty}{\operatorname{cap}}

\providecommand{\abs}[1]{\lvert#1\rvert}
\providecommand{\norm}[1]{\lVert#1\rVert}

\renewcommand{\epsilon}{\varepsilon}
\renewcommand{\phi}{\varphi}

\newtheoremstyle{Beispiel}{}{}{}{}{\bfseries}{:}{ }{}
\theoremstyle{Beispiel}
\newtheorem{example}{Example}[section]
\newtheorem*{remark}{Remark}

\newtheoremstyle{Satz}{}{}{\itshape}{}{\bfseries}{:}{ }{}
\theoremstyle{Satz}
\newtheorem{proposition}[example]{Proposition}
\newtheorem{definition}[example]{Definition}
\newtheorem{theorem}[example]{Theorem}
\newtheorem{lemma}[example]{Lemma}
\newtheorem{corollary}[example]{Corollary}
\newtheorem*{theorem*}{Theorem}

\newcommand{\Hmm}[1]{\leavevmode{\marginpar{\tiny%
$\hbox to 0mm{\hspace*{-0.5mm}$\leftarrow$\hss}%
\vcenter{\vrule depth 0.1mm height 0.1mm width \the\marginparwidth}%
\hbox to 0mm{\hss$\rightarrow$\hspace*{-0.5mm}}$\\\relax\raggedright
#1}}}

\begin{document}
\frenchspacing 

\title[Uniqueness of form extensions]{Uniqueness of form extensions and domination of semigroups}

\author[Lenz]{Daniel Lenz}
\address{D. Lenz, Mathematisches Institut\\Friedrich-Schiller-Universität Jena\\07737 Jena, Germany}
\email{daniel.lenz@uni-jena.de}

\author[Schmidt]{Marcel Schmidt}
\address{M. Schmidt, Mathematisches Institut\\Friedrich-Schiller-Universität Jena\\07737 Jena, Germany}
\email{schmidt.marcel@uni-jena.de}

\author[Wirth]{Melchior Wirth}
\address{M. Wirth, Mathematisches Institut\\Friedrich-Schiller-Universität Jena\\07737 Jena, Germany}
\address{Address at the time of publication: IST Austria, Am Campus 1, 3400 Klosterneuburg, Austria}
\email{melchior.wirth@ist.ac.at}

\date{\today}

\begin{abstract}
In this article, we study uniqueness of
form extensions in a rather general setting. The method is based on
the theory of ordered Hilbert spaces and the concept of domination
of semigroups. Our main abstract result  transfers uniqueness of
form extension of a dominating form to that of a dominated form.
This result can be applied to  a multitude of examples including
various magnetic Schr\"odinger forms on graphs and on manifolds.
\end{abstract}

\maketitle




\section*{Introduction}
In many  situations in mathematics and physics one is given a
Laplace-type operator on a domain of smooth functions and looks for
self-adjoint extensions. Existence of such self-adjoint extensions
follows from general theory of forms. Indeed, there always exists an
extension with Dirichlet boundary condition and an extension with
Neumann boundary condition. Accordingly, these are the most common
extensions. Now, it may well be that these two extensions agree and
the question whether this happens is of quite some interest.  In
this context we also  mention the even stronger  property of
essential self-adjointness, i.\,e., uniqueness of a self-adjoint
extension studied extensively  on manifolds, see e.g.
\cite{LT95,Mas99,BMS02,Sch01}, and on graphs, see e. g.
\cite{CTT11a,CTT11b,HKMW,HL15,KL12,KS03,Woj08,Woj09}.

On the structural level, the question whether the Laplacian with
Dirichlet and the Laplacian with  Neumann boundary conditions agree
is connected to an additional feature of the Laplacian, namely that it
may generate a Markov semigroup. More specifically, in many
situations all self-adjoint extensions of the Laplacian which
generate a Markov semigroup  can be shown to lie between the
Laplacian with Dirichlet boundary conditions and the Laplacian with
Neumann boundary conditions (see \cite{Fu} for open subsets of
Euclidean space, \cite{HKLW}  for locally finite graphs and
\cite{Schm16} for recent results dealing with general Dirichlet
forms). Thus, equality of these two boundary conditions then amounts
to uniqueness of a self-adjoint extension generating a Markov
process. This phenomenon is known as Markov uniqueness. Clearly, it
is of substantial interest in any operator theoretic treatment of
Markov processes and has therefore  received ample attention, see
e.g. \cite{Ebe99,KT96,Kuw02,KuShi15,RS11,Rob13,Silv74,Sim79,GM13}.

The corresponding questions can be asked for Laplacians on functions
as well as for the more involved Laplacians on bundles or with
magnetic or electric potential. In fact, recent years have seen
quite a few articles (e.g. \cite{BMS02,CTT11b,MT14,MT15}) dealing
with uniqueness questions for extensions of Laplacians with magnetic
potential or Laplacians on bundles over graphs and manifolds.

In this paper we present a new approach to studying  equality of
Dirichlet and Neumann Laplacian on bundles or with magnetic
potential. Our approach even  deals with uniqueness questions  in a
substantially  more general context. Its key element is to consider
the question of uniqueness within the framework of dominating
semigroups. The study of domination of semigroups has a long
history. A treatment of domination  in terms of a Kato inequality
for the generators of the semigroups was given  in the influential
works of Simon \cite{Sim79b} and  Hess, Schrader, Uhlenbrock \cite{HSU77}. Here,
\cite{Sim79b} deals with comparison of two semigroups acting on the
same $L^2$-space and \cite{HSU77} deals with comparison of two
semigroups acting on rather general and even possibly different
Hilbert spaces.  For semigroups acting on the same $L^2$-space a
characterization of domination  via forms was later developed by
Ouhabaz in \cite{Ouh96,Ouh99} and then extended to forms acting on
vector-valued functions by Manavi, Vogt, Voigt \cite{MVV}.  For our
considerations we rely on  the framework developed by us  in the
companion work \cite{LSW19}. This framework can be seen as an
abstraction of \cite{MVV}  to the setting of rather general Hilbert
spaces originally studied in \cite{HSU77}.

Our main abstract theorem (Theorem \ref{uniqueness_forms}) gives
stability of uniqueness under a domination condition. To the best of
our knowledge it is the very first result of its kind. It is highly
relevant as often the question of uniqueness of form extensions is
easier or already proven for the dominating form.

As a corollary we prove that, roughly speaking, Dirichlet and
Neumann Laplacian on a bundle  agree whenever the corresponding
Laplacians on the underlying space agree (Corollaries
\ref{corollary_uniqueness_forms} and  \ref{uniqueness_smooth}). This
type of result can then be applied to various examples including
magnetic Schr\"odinger operators on manifolds and on graphs (see
Sections \ref{Applications-mfld} and \ref{Applications-grph}). In
particular, we recover in a new and direct way  all known examples
for graphs and domains in Euclidean space and provide new examples
which were not treated earlier. Another application of our main
theorem to Kac regularity of domains has been given in a recent
preprint \cite{Wir17} by the third-named author.

In retrospect it is not completely surprising that domination of
semigroups plays a role in investigation of uniqueness issues.
Indeed, as already discussed above, domination of semigroups is
equivalent to the validity of a variant of Kato's inequality for the
generators  and this inequality is a key element in all previous
proofs of essential self-adjointness of Laplacians with magnetic
potential. However, even in those cases where this was shown
earlier, the actual line of reasoning is quite different from ours:
It proceeds by treating the magnetic situation by mimicking the
proof given for the Laplacian and invoking Kato' inequality. In this
sense, our paper provides a conceptual connection between
uniqueness of extensions and the theory of dominating semigroups.
Note, however, that our result does not deal with essential
self-adjointness but rather with a somewhat weaker statement that
can be  thought of as a form of Markov uniqueness (see discussion
above). In the final analysis the reason that our methods do not
give stability of essential self-adjointness may come from the fact
that domination of semigroups concerns the order structure and can
therefore only be expected to be of use in connection with
extensions respecting some form of order such as extensions that generate
Markov semigroups.

The framework of this article are symmetric forms and selfadjoint
operators as this is the situation in the applications we have in
mind. However, it is to be expected that sectorial forms and their
generators  can be treated by similar means as the
domination theory developed in \cite{MVV} works for such forms.

The article is organized  as follows: In Section
\ref{chap_positivity} we review the basics  on quadratic forms on
Hilbert spaces and domination of operators. In Section
\ref{chap_uniqueness_criterion} we prove the main theorem of this
article (Theorem \ref{uniqueness_forms}), a criterion for form
uniqueness in terms of domination. In the subsequent  two sections
we discuss the above mentioned applications, namely Schr\"odinger
forms on vector bundles on manifolds in Section \ref{Applications-mfld} and
Schr\"odinger forms on vector bundles over graphs in  Section
\ref{Applications-grph}. In Appendix~\ref{appendix} we fill a gap in the literature by characterizing a property that yields form uniqueness on graphs. 

The article has its origin in the master's thesis of one of the
authors (M.~W.).

\subsection*{Acknowledgments} M.~S. and M.~W. gratefully acknowledge
financial support of the DFG via \emph{Graduiertenkolleg: Quanten-
und Gravitationsfelder}. M.~W. gratefully acknowledges financial
support of the \emph{Studienstiftung des deutschen Volkes}. D.~L.
gratefully acknowledges partial support by the DFG as well as
enlightening discussions with Peter Stollmann on Dirichlet forms and
domination of semigroups. M.~W. would like to thank Ognjen Milatovic
for several helpful remarks on a preliminary version of this article
and Jun Masamune for fruitful discussions about possible
applications of the main result.

\section{Quadratic forms and domination of semigroups}\label{chap_positivity}
Our considerations are cast in the framework of semigroups
associated to lower boun\-ded forms.  The main abstract result of
this article (presented  in the next section) gives stability of
form uniqueness under domination. Here domination is meant in the
sense of \cite{LSW19} (see also the discussion below) and  form uniqueness will be phrased via
suitable cores of the form. Specifically, these cores will have an
ideal property. A crucial role is played by forms on $L^2$-spaces
satisfying the first Beurling-Deny criterion.  The necessary
background and notation is discussed in this section. If not noted
otherwise the material can be found in standard textbooks such as
\cite{Wei,RS}.

\medskip

Let $H$ be a complex Hilbert space. A quadratic form $q\colon D(q)\lra \IR$
defined on a subspace $D(q)$ of $H$ induces a sesquilinear form
\begin{align*}
D(q)\times D(q)\lra \IC,\,(f,g)\mapsto \frac 1 4\sum_{k=1}^4 i^k q(f+i^k g)
\end{align*}
We will not distinguish between a quadratic form and the induced
sesquilinear form and write $q$ for both of them. Moreover, all quadratic forms in this article are assumed to be densely defined.

The form $q$ is said to be \emph{lower bounded} if there exists
$\lambda\in \IR$ such that $q(f)\geq -\lambda\norm{f}_H^2$ for all
$f\in D$. In this case the form $q_\mu := q + \mu \langle
\cdot,\cdot \rangle$ is an inner product on $D(q)$ and the norms
associated to $q_{\mu}$, $q_{\mu'}$ are equivalent for
$\mu,\mu'>\lambda$. If the choice of $\mu$ does not matter, we will
write $\langle \cdot,\cdot\rangle_q$ for $q_\mu$ and
$\norm{\cdot}_q$ for the associated  norm and call it the \emph{form
norm} of $q$. A lower bounded form $q$ is \emph{closed} if $(D(q), \langle
\cdot,\cdot\rangle_q)$ is complete (and, hence, a Hilbert space). A dense subspace of $(D(q),\norm\cdot_q)$ is also called a \emph{core} for $q$.

For every closed form $q$ there exists a unique self-adjoint
operator $T$ with $D(T) \subset D(q)$ and   $\langle T
f,g\rangle=q(f,g)$ for all $f\in D(T)$, $g\in D(T)$. The domain of
$T$ is then dense in $D(q)$ with respect to $\norm{\cdot}_q$. The
operator $T$ is called the \textit{generator} of $q$ (note that this
convention makes the generator a \emph{lower} semibounded operator).

Forms on $L^2$-spaces and, in particular,  those well compatible
with its lattice structure will be of main relevance for our study.
These are discussed next. Let $(X,\mathcal{B},\mu)$ be a measure
space. A form $q$ on $L^2(X,\mu)$ is \emph{real} if $f\in D(q)$
implies $\bar f\in D(q)$ and $q(\bar f)=q(f)$. Of course this
condition is void when working over $\IR$. For real functions $f,g$
we define
$$f\wedge g := \min\{f,g\}.$$   A real form $q$ satisfies the \emph{first Beurling-Deny criterion}
if $f\in D(q)$ implies $\abs{f}\in D(q)$ and $q(\abs{f})\leq q(f)$.
In this case also $q(f^+)\leq q(f)$ for any real valued $f$, where
$f^+$ denotes the positive part of $f$ i.e. $f^+:=\max\{f,0\}$. The
importance of the first Beurling-Deny criterion comes from the fact
that a  form $q$ with generator $T$ satisfies the first Beurling-Deny criterion if and only if the associated semigroup $(e^{-t T})_{t\geq 0}$
is \textit{positivity preserving}, that is, $e^{-t T} f\geq 0$
for all $t\geq 0$ and  $f\geq 0$.

Domination will provide the main condition for our stability result.
The relevant aspects of domination  will be  discussed next (see
\cite{HSU77,MVV,Ouh96,Ouh99,Sim79b} as well for further discussion).
The basic idea is to 'estimate'  a form or semigroup on a Hilbert
space $\H$ by a form or semigroup on an $L^2$-space. In order to
achieve this comparison one needs a map from $\H$ to the $L^2$-space
mimicking the modulus of a complex number and allowing for a version
of the polar decomposition of a complex number as well. Following
\cite{HSU77} such maps  will be discussed next.

Let $(X,\mathcal{B},\mu)$ be a measure space and $\H$ a Hilbert
space. We denote by $L^2_+(X,\mu)$ the space all nonnegative
functions in $L^2(X,\mu)$ and, more generally, whenever  $V\subset
L^2(X,\mu)$ is a subspace, we write $V_+$ for $V\cap L^2_+(X,\mu)$.

A map $\abs{\cdot}\colon \H\lra L^2_+(X,\mu)$ is called
\emph{absolute map} if
\begin{align*}
\abs{\langle \xi,\eta\rangle}\leq \langle\abs{\xi},\abs{\eta}\rangle
\end{align*}
for all $\xi,\eta\in \H$ with equality if $\xi=\eta$. So, in
particular, $|\xi| =0$ implies $\xi =0$.  The elements $\xi$ and
$\eta$ of $\H$ are called \emph{paired} if
\begin{align*}
\langle \xi,\eta\rangle=\langle\abs{\xi},\abs{\eta}\rangle.
\end{align*}
 If for all $\xi\in
\H$ and $f\in L^2_+(X,\mu)$ there exists $\eta\in \H$ such that
$\abs{\eta}=f$ and $\xi$ and $\eta$ are paired the absolute map is
called \emph{absolute pairing}.

Clearly, absolute maps are generalizations of the modulus function.
Similarly, an absolute pairing can be thought of as giving not only
a modulus function but also a polar decomposition. More
specifically, whenever  $\xi$ and $\eta$ are paired one may think
about $\eta$ as being of the form $ \abs{\eta} \sgn \xi$. In fact,
the  reader may just bear the following examples in mind, where such
an interpretation is rather immediate.

\begin{example}[Direct integral] \label{direct-integral}
Let $((H_x)_{x\in X},\mathfrak{M})$ be a measurable field of Hilbert
spaces over $(X,\mathcal{B},\mu)$ in the sense of \cite{Tak02},
Definition IV.8.9, and $\H=\int_X^\oplus H_x\,d\mu(x)$. The norm on
$H_x$ is denoted by $\abs{\cdot}_x$, $x\in X$.
 Then, the
map $\abs{\cdot}\colon \H\lra L^2_+(X,\mu)$ given by
$\abs{\xi}(x)=\abs{\xi(x)}_x$ is an absolute pairing. Indeed, for
$\xi\in\H$ and $f\in L^2_+(X,\mu)$ let
\begin{align*}
\eta(x)=\begin{cases}\frac{ f(x)}{\abs{\xi(x)}_x}  \xi(x) &\colon \xi(x)\neq 0,\\
f(x)\zeta(x)&\colon \xi(x)=0,\end{cases}
\end{align*}
where $\zeta\in\mathfrak{M}$ with $\abs{\zeta}=1$ (the existence of
such an element is proven in \cite{Tak02}, Lemma IV.8.12). Then
$\eta$ and $\xi$ are paired with $\abs{\eta}=f$. The other
properties of an absolute pairing are easy to check.

A special case of this construction is given by a constant field of
Hilbert spaces, where all  $H_x$, $x\in X$,  are equal to one fixed
separable Hilbert space $H$. In this case, $\H$ is also denoted by
$L^2 (X,\mu;H)$ and given by the vector space of all measurable
maps (with respect to the corresponding Borel-$\sigma$-algebras)
$\xi : X\longrightarrow H$ with $\int \abs{\xi(x)}_x^2 d\mu (x)
<\infty$, where two such maps are identified if they agree
$\mu$-almost everywhere.
\end{example}

\smallskip

For the applications discussed in the later part of this article, we
will  restrict attention to  the  special instance of the previous
example given by  Hermitian vector bundles. There the Hilbert spaces
in question are finite dimensional and form a continuous field. For
definiteness reasons we provide an explicit discussion next.

\begin{example}[Hermitian bundle] \label{hermitian-bundle}
A Hermitian vector bundle is given by  topological spaces $E$ and
$X$ and a continuous surjective map $\pi : E\longrightarrow X$
 with the property that  each fiber $\pi^{-1} (x)$, $x\in X$, carries the structure of a
finite-dimensional complex vector space with an inner product $\langle
\cdot,\cdot\rangle_x$  and  that there is a  finite-dimensional complex
Hilbert space $(H,\langle\cdot,\cdot\rangle)$ such that to each
point $p\in X$ there exists a neighborhood $U$ and a homeomorphism
$\varphi : U\times H\longrightarrow \pi^{-1} (U)$, called local
trivialization, such that $\varphi_x :=\varphi (x,\cdot)$ is an an
isometric isomorphism between the inner product space $H$ and
$\pi^{-1}(x)$ for each $x\in U$. We will  denote such  a bundle  by
$\pi\colon E\lra X$ and  say that $E$ is a bundle over $X$ (with
projection $\pi$). We subsequently also often  skip the $\pi$ and
the $X$ in the notation and  just  call $E$ the bundle. (For further
details we refer to \cite{MS74}, where the real-valued version, \textit{Euclidean
vector bundles}, are discussed.)

Let now $\pi \colon E\lra X$ be a Hermitian vector bundle,  $\mu$ a
Borel measure on $X$ and assume that $X$ is  a Lindelöf space, that
is  a topological space such that every open cover has a countable
subcover. Examples of Lindelöf spaces include $\sigma$-compact
spaces and separable metric spaces.  A function $\xi :
X\longrightarrow E$ with $\pi \circ \xi = \mathrm{id}_X$ is called a
section. By $L^2 (X,\mu;E)$ we denote the set of measurable (with
respect to the corresponding Borel-$\sigma$-algebras) sections $\xi$
with $\int_X \abs{\xi(x)}_x^2   dm (x) < \infty$, where
$\abs{\cdot}_x$ is the norm induced from $\langle
\cdot,\cdot\rangle_x$ and sections are identified which agree
$\mu$-almost everywhere. This space is a Hilbert space and
\begin{align*}
\abs\cdot\colon L^2(X,\mu;E)\lra L^2_+(X,\mu),\,\abs{\xi}(x)=\abs{\xi(x)}_x
\end{align*}
is an absolute pairing.

Indeed, this is just a special case of the
previous example with constant field of Hilbert spaces. More precisely, by the Lindelöf property there is a countable family of local trivializations $\phi_k\colon U_k\times H\lra \pi^{-1}(U_k)$ such that $X=\bigcup_k U_k$. Setting $V_k=U_k\setminus\bigcup_{j=1}^{k-1} U_j$, one obtains
\begin{align*}
\psi\colon X\times H\lra E,\,\psi|_{V_k}=\phi_k|_{V_k}
\end{align*}
which is a measurable bijection with measurable inverse. By definition, the induced map
\begin{align*}
\Psi\colon L^2(X,\mu;H)\lra L^2(X,\mu;E),\,\Psi(\xi)(x)=\psi(x,\xi(x))
\end{align*}
is an isometric isomorphism and $\abs{\Psi(\xi)}=\abs{\xi}$ for all $\xi\in L^2(X,\mu;H)$.
\end{example}

Let now an absolute pairing  $\abs{\cdot}\colon \H\lra L^2_+(X,\mu)$
be given. Then, the subspace $U$ of $\H$ is
 a \emph{generalized ideal} of the subspace  $V$  of $L^2 (X,\mu)$
 if for all
$\xi\in U$ we have $\abs{\xi}\in V$ and  for any $f\in V$ with
$0\leq f\leq \abs{\xi}$ there exists an $\eta\in U$ such that
$\abs{\eta} = f$ and $\xi$ and $\eta$ are paired.

Let $\a$ be a closed form on $\H$ and  $\b$ a closed form on
$L^2(X,\mu)$. We say that \emph{$\a$ is dominated by $\b$} if
$D(\a)$ is a generalized ideal of $D(\b)$ and
\begin{align*}
\Re \a(\xi,\eta)\geq \b(\abs{\xi},\abs{\eta})
\end{align*}
whenever $\xi,\eta\in D(\a)$ are paired. In the situations we have
in mind the dominating form $\b$ will  additionally satisfy the
first Beurling-Deny criterion.  The relevance of domination comes
from the following fundamental result.

\smallskip

\begin{theorem*}
Let $A$, $B$ be the generators of the forms $\a$ and $\b$
respectively and assume that $\b$ satisfies the first Beurling-Deny
criterion. Then, the form $\a$ is dominated by $\b$ if and only if
the semigroup $(e^{-tA})_{t\geq 0}$ is dominated by
$(e^{-tB})_{t\geq 0}$ in the sense that
\begin{align*}
\abs{e^{-tA}\xi}\leq e^{-tB}\abs{\xi}
\end{align*}
for all $\xi\in \H$ and all $t\geq 0$.
\end{theorem*}

\smallskip

 This result has quite some history: For $\H =
L^2 (X,\mu)$ is was shown in \cite{Ouh96,Ouh99}. This was then
extended to  $\H=L^2(X,\mu;H)$ for a Hilbert space $H$ in \cite{MVV}
(where even sectorial forms are treated). The general case stated
here (and actually an even more general case) is given in
\cite{LSW19}. For the applications in the later part of the article
the result of \cite{MVV} suffices.

%
%
%

As discussed already, our aim is to study uniqueness of form
extensions. So, we will be interested in cores of quadratic forms. These cores
will need to have a special structure which we introduce next. Let
$\abs{\cdot}\colon \H\lra L^2_+(X,\mu)$ be an absolute pairing and
$U$, $V$ subspaces of $\H$. The space $U$ is called an \emph{ideal}
of $V$ if $f\in U$, $g\in V$ and $\abs{g}\leq \abs{f}$ implies $g\in
U$. If $U,V\subset L^2(X,\mu)$, this is meant with respect to the
absolute pairing given by the pointwise modulus. So, in this case
our notion of ideals coincides with the usual definition in the
theory of Banach lattices.

\begin{remark} The concepts of ideal and  of generalized ideal
certainly have the same flavor. So, it is worth noting that  there
is no general relation between these two concepts.  Indeed, they
arise in rather  different situations. Ideals come about as
subspaces of the same Hilbert space, whereas the notion of generalized ideals applies to subspaces of two different Hilbert spaces, which are linked by an absolute pairing. So, in this respect the terminology is somewhat
unsatisfactory. Still, we stick with it, as it seems to be the
standard notation used in the field. Connections between ideals and generalized ideals in the case $\H=L^2(X,\mu)$ are studied in \cite{MVV}, Proposition 3.6 and Corollary 3.7.
\end{remark}

\section{A criterion for form uniqueness}\label{chap_uniqueness_criterion}
In this section we present the main theorem of this article, which
allows one to transfer form uniqueness of a dominating form to that
of the dominated form.

\medskip

We begin with two technical lemmas that may be of interest in other
situations as well. The first lemma shows that the form norm is
compatible with taking minima.

\begin{lemma}\label{pos_form_lattice}
Let $(X,\mathcal{B},\mu)$ be a measure space and $q$ be a lower bounded quadratic form on $L^2(X,\mu)$ satisfying the first Beurling-Deny
criterion. If $f,g\in D(q)$ are real-valued, then $f\wedge g\in D(q)$ and
\begin{align*}
\norm{f\wedge g}_q^2\leq \norm{f}_q^2+\norm{g}_q^2.
\end{align*}
\end{lemma}
\begin{proof}
Since $\norm{\abs{f}}_{L^2}=\norm{f}_{L^2}$, the square of the form
norm $\norm{\cdot}_q^2=q+\mu\norm{\cdot}_{L^2}^2$ also satisfies the
first Beurling-Deny criterion. Also note that $f\wedge g=\frac 1
2(f+g-\abs{f-g})\in D(q)$. Combining these two facts with Young's
inequality, we obtain
\begin{align*}
\norm{f\wedge g}_q^2&=\frac 14\norm{f+g-\abs{f-g}}_q^2\\
&\leq\frac 1 2(\norm{f+g}_q^2+\norm{\abs{f-g}}_q^2)\\
&\leq \frac 1 2(\norm{f+g}_q^2+\norm{f-g}_q^2)\\
&=\norm{f}_q^2+\norm{g}_q^2.\qedhere
\end{align*}
\end{proof}

The next lemma gives an approximation result.

\begin{lemma}\label{approx_pos}
Let $(X,\mathcal{B},\mu)$ be a measure space, $q$ a closed form on $L^2(X,\mu)$ satisfying the first Beurling-Deny
criterion, and $D_q\subset D(q)$ a dense ideal. If $v\in D(q)_+$,
then there is a sequence $(v_n)$ in $D_q$ such that $0\leq v_n\leq
v$ and $v_n\to v$ with respect to $\norm{\cdot}_q$ and almost everywhere.
\end{lemma}
\begin{proof}
Let $v\in D(q)_+$. Since $D_q\subset D(q)$ is dense, there is a
sequence $\tilde v_n$ in $D_q$ such that
$$
\norm{\tilde v_n-v}_q\to
0
$$
and $\tilde v_n \to v$ pointwise a.e. Since $D_q$ is closed under taking real parts, we may assume that the functions $\tilde v_n$ are real-valued.

Lemma   \ref{pos_form_lattice} ensures $\tilde v_n^+\wedge v\in
D(q)$ and
\begin{align*}
\norm{\tilde v_n^+\wedge v}_q^2\leq \norm{\tilde v_n^+}_q^2+\norm{v}_q^2\leq \norm{\tilde v_n}_q^2+\norm{v}_q^2.
\end{align*}
From the inequality
$$0\leq \tilde v_n^+\wedge v\leq \tilde v_n^+\leq \abs{\tilde v_n}$$
it follows that $\tilde v_n^+\wedge v$ belongs to $ D_q$ since
$\tilde v_n$ belongs to $D_q$,  $\tilde v_n^+\wedge v$ belongs to
$D(q)$ and $D_q\subset D(q)$ is an ideal.

Since $(\tilde v_n)$ is convergent in
$(D(q),\langle\cdot,\cdot\rangle_q)$, it is in particular bounded,
and the above  inequality shows that $(\tilde v_n^+\wedge v)$ is
bounded as well. By the Banach-Saks Theorem (cf. \cite{CF}, Theorem~A.4.1) there is a subsequence $(\tilde v_{n_k})$ and an element
$\tilde v\in D(q)$ such that
\begin{align*}
v_N:=\frac 1 N\sum_{k=1}^N \tilde v_{n_k}^+\wedge v\overset{\norm\cdot_q}{\to}\tilde v,\,N\to\infty.
\end{align*}
Obviously, $v_N\in D_q$ and $0\leq v_N\leq v$ for all $N\in \IN$. By the pointwise a.e. convergence of $\tilde v_n$ to $v$ also the Ces\`{a}ro means
$v_N$ converge to $v$ a.e. Hence $\tilde v=v$.
\end{proof}

Given the previous lemma we can now rather easily prove   the main
abstract result of the article. If $\abs{\cdot}\colon \mathcal{H}\lra L^2_+(X,\mu)$ is an absolute pairing and $V\subset \mathcal{H}$ a subspace, we write $\abs{V}$ for $\{\abs{v}: v\in V\}$.

\begin{theorem}\label{uniqueness_forms}
Let $\H$ be a Hilbert space, $(X,\mathcal{B},\mu)$ a measure space,
and $\abs{\cdot}\colon \H\lra L^2_+(X,\mu)$ an absolute pairing. Let
$\a$ be a closed form in $\H$, $\b$ a closed form in $L^2(X,\mu)$ satisfying
the first Beurling-Deny criterion, and $D_\a\subset D(\a), D_{\b}\subset D(\b)$
ideals such that the following conditions hold:
\begin{itemize}
\item $\a$ is dominated by $\b$
\item $D_\b^+\cap \abs{D(\a)}\subset \abs{D_\a}$
\end{itemize}
If $D_\b$ is a core for $\b$, then $D_\a$ is a core for $\a$.
\end{theorem}
\begin{proof}
To ease notation, we assume without loss of generality that $\a$ and $\b$ are positive. As $\a$
is closed, $D(\a)$ is a Hilbert space with the inner product
$\langle\cdot,\cdot\rangle_\a=\langle\cdot,\cdot\rangle_{\H}+\a(\cdot,\cdot)$
and analogously for $\b$.

\smallskip

We show that $D_\a\subset D(\a)$ is dense with respect to
$\norm\cdot_\a$ by proving that $D_\a^\perp=\{0\}$ in
$(D(\a),\langle\cdot,\cdot\rangle_\a)$. For this purpose, let $h\in
D(\a)$ such that
\begin{align*}
0=\langle h,u\rangle_\a=\langle h,u\rangle+\a(h,u)
\end{align*}
for all $u\in D_\a$.

\smallskip

Consider $v\in D_\b^+$ such that $v\leq \abs{h}$. Since $\a$ is
dominated by $\b$, $D(\a)$ is a generalized ideal of $D(\b)$. Hence,
as $h$ belongs to $D (\a)$,  there exists $\tilde h\in D(\a)$ such
that $\abs{\tilde h}=v$ and $h,\tilde h$ are paired. This implies in
particular,
$$v= \abs{\tilde h} \in D_\b^+\cap \abs{D(\a)}\subset \abs{D_\a},$$
 where we used the assumption on $D_\b$ for the last
inclusion. Since $D_\a$ is an ideal in $D(\a)$ we then obtain
$\tilde h\in D_\a$. By the orthogonality assumption on $h$ above
this implies
$$0 = \langle h, \tilde h\rangle_\a.$$
Now, as  $\a$ is dominated by $\b$, we have $\abs{h}\in D(\b)$ and
from the preceding equality and domination we infer
\begin{align*}
0=\langle h,\tilde h\rangle+\Re\a(h,\tilde h)\geq \langle \abs{h},v)\rangle+\b(\abs{h},v).\tag{$\ast$}\label{ineq}
\end{align*}
As
$D_\b$ is a core for $\b$,  Lemma \ref{approx_pos} can be applied
with $v = \abs{h}$ and  there exists a sequence $(v_n)$ in $D_\b$
such that $0\leq v_n\leq \abs{h}$ and $\norm{v_n-\abs{h}}_\b\to 0$.
Applying inequality (\ref{ineq}) to $v= v_n$ we obtain
\begin{align*}
0\geq \langle \abs{h}, v_n\rangle+\b(\abs{h}, v_n)=\langle \abs{h}, v_n\rangle_\b\to \norm{\abs{h}}_\b^2.
\end{align*}
Hence $\abs{h}=0$ and therefore also $h=0$. Thus, $D_\a^\perp=\{0\}$.
\end{proof}

\begin{remark}
\begin{itemize}
\item In applications, the situation will often be as follows: We are
given forms $\a_0$ on $D_\a$, $\b_0$ on $D_\b$ (usually not closed)
and minimal extensions $\a_{\text{min}}$, $\b_{\text{min}}$ (the
closures of $\a_0$, $\b_0$) and maximal extensions
$\a_{\text{max}}$, $\b_{\text{max}}$. If
$\b_{\text{min}}=\b_{\text{max}}$, then the theorem yields
$\a_{\text{min}}=\a_{\text{max}}$. This situation is discussed in
detail in the subsequent two sections.

\item If $\H=L^2(X,\mu)$, $\abs\cdot$ is the pointwise modulus, and
$D_\a=D_\b$, then the condition $D_ \b ^+\subset \abs{D_\a}$ is
automatically satisfied.
%
\end{itemize}
\end{remark}

We now turn to  a  corollary that contains the concrete situation of
our applications in the next sections.  There, we consider a
Lindelöf space $X$ and  $\mu$ a Borel measure on $X$ and a
Hermitian vector bundle $E$ over $X$ (compare Example
\ref{hermitian-bundle} above for details). In this situation, we
denote by $L^2_c(X,\mu)$ the space of square integrable functions
that vanish outside a compact set and  by $L^2_c(X,\mu;E)$ the space
of square integrable sections in $E$ that vanish outside a compact
set.

\begin{corollary}\label{corollary_uniqueness_forms}
Let $X$ be a Lindelöf space, $\mu$ a Borel measure on $X$ and
 $E$  a Hermitian vector bundle over $X$. Assume that $\b$ is a closed form in $L^2(X,\mu)$ satisfying the first Beurling-Deny criterion and $\a$ a closed form in $L^2(X,\mu;E)$ that is dominated by $\b$.
If $D(\b)\cap L^2_c(X,\mu)$ is a core for $\b$, then $D(\a)\cap L^2_c(X,\mu;E)$ is a core for $\a$.
\end{corollary}
\begin{proof}
We will apply Theorem \ref{uniqueness_forms} to $D_\a=D(\a)\cap L^2_c(X,\mu;E)$ and $D_\b=D(\b)\cap L^2_c(X,\mu)$. It is obvious that these are ideals in $D(\a)$ and $D(\b)$ respectively.\\
Now let $g\in D_\b^+\cap \abs{D(\a)}$. Then there is an $f\in D(\a)$ such that $\abs{f}=g\in L^2_c(X,\mu)$. Thus, $f\in D(\a)\cap  L^2_c(X,\mu;E)$ and $g=\abs{f}\in \abs{D_\a}$.
\end{proof}

\begin{remark}
\begin{itemize}
\item If $\b$ is a regular Dirichlet form, $D(\b)\cap
L^2_c(X,\mu)$ is a core for $\b$. Indeed, $D(\b)\cap C_c(X)\subset
D(\b)\cap L^2_c(X,\mu)$ is dense in $D(\b)$ by definition.
Note, however, that we do not use the second Beurling-Deny criterion
in our reasoning.
\item The preceding corollary concerns  bundles $E$
over an  underlying topological space $X$. As discussed in Example
\ref{hermitian-bundle}, the  space of the $L^2$-sections in the
bundle can also be considered as a direct integral of Hilbert spaces
over $X$. In fact, it is easily possible to  generalize the
corollary to the setting of  direct integrals discussed in Example
\ref{direct-integral}, but we do not need this for the purposes of
this article.
\end{itemize}
\end{remark}

In applications to manifolds one is in an even more regular
situation. More specifically, in the smooth case, one is usually
interested in the closure of the form defined on smooth functions
(sections) as minimal form. We make the following definition adapted
to this situation.

\begin{definition}\label{smoothly_inner_reg}
Let $M$ be a Riemannian manifold and $E\lra M$ a smooth Hermitian
vector bundle and denote by $\Gamma_c(M;E)$ the space of compactly supported smooth sections in $M$. A form $\a$ on $L^2(M;E)$ is called \emph{smoothly
inner regular} if $D(\a)\cap \Gamma_c(M;E)$ is dense in
$D(\a)\cap L^2_c(M;E)$ with respect to $\norm\cdot_\a$.
\end{definition}

From the definition of smooth inner regularity and the above
corollary, the following corollary can easily be deduced.

\begin{corollary}\label{uniqueness_smooth}
Let $M$ be a Riemannian manifold and $E\lra M$ a smooth Hermitian
vector bundle. Let $\b$ be a closed form on $L^2(M)$ satisfying the first Beurling-Deny criterion and $\a$ a
closed, smoothly inner regular form on $L^2(M;E)$ that is dominated
by $\b$. If $D(\b)\cap C_c^\infty(M)$ is a core for $\b$, then
$D(\a)\cap\Gamma_c(M;E) $ is a core for $\a$.
\end{corollary}

\section{Schr\"odinger forms on weighted Riemannian manifolds}\label{Applications-mfld}

In this section we study quadratic forms associated to Schrödinger operator on vector bundles over Riemannian manifolds. This kind of operators and the associated forms have been extensively studied in the last decades, let us just mention \cite{BG20,BMS02,Gue14,HSU80} as references covering all necessary basics for this section.

After introducing the quadratic forms in question, we prove that the Schrödinger forms on vector bundles are dominated by the corresponding forms acting on functions (Proposition \ref{domination_manifolds}), which implies the uniqueness result in this setting (Theorem \ref{uniqueness_Schrödinger_manifolds}). Finally we discuss how this result enables us to apply capacity estimates to uniqueness questions for Schrödinger forms on vector bundles (Corollary \ref{capacity_estimates}).

Throughout this section let $(M,g,\mu)$ be a weighted
Riemannian manifold, that is, $(M,g)$ is a Riemannian manifold and
$\mu=e^{-\psi}\vol_g$ for some $\psi\in C^\infty(M)$. All (local)
Lebesgue and Sobolev spaces are taken with respect to the measure
$\mu$.

\begin{definition}[Regular Schrödinger bundle \cite{BG20}]
A \emph{regular Schr\"odinger bundle} is a triple
$(E,\nabla,W)$ consisting of
\begin{itemize}
\item a smooth Hermitian vector bundle $E$ over $M$,
\item a metric covariant derivative $\nabla$ on $E$,
\item a potential $W\in L^1_\loc(M;\End(E))_+$.
\end{itemize}
\end{definition}

Here $\End(E)$ denotes the set of bundle endomorphisms of $E$, that is, the smooth maps from $E$ to itself that restrict to linear maps on the fibers. It can be made into a vector bundle in such a way that the fiber over $x$ is the space of all linear maps from $E_x$ to itself. As usual, $L^1_{\loc}(M;\End(E))$ denotes the set of all locally integrable sections in the vector bundle $\End(E)$, and the subscript $+$ indicates the subset of sections which are positive linear operators on $E_x$ for a.e. $x\in M$.

For a vector bundle $E\lra M$ we denote by $\Gamma(M;E)$ the space
of smooth sections and by $\Gamma_c(M;E)$ the subspace of compactly
supported smooth sections. If $E\lra M$ is Hermitian, we write
$\langle\cdot|\cdot\rangle$ and $\abs\cdot$ for the inner product
and induced norm on the fibers, respectively.

%

\begin{example}
If $\eta\in\Gamma(M;T^\ast M)$, then
$\nabla=d+i\eta$ is a metric covariant derivative on the trivial
complex line bundle $M\times\IC\lra M$. Thus, magnetic Schrödinger
operators with electric potential are naturally included in this
setting.
\end{example}

\begin{definition}[Schrödinger form with Neumann boundary conditions]
Let
\begin{align*}
W^{1,2}(M;E)&=\{\Phi\in L^2(M;E)\mid \nabla \Phi\in L^2(M;E\otimes
T^\ast_\IC M)\},
\end{align*}
where $\nabla$ is to be understood in the distributional sense. The
space $W^{1,2}_\loc(M;E)$ is defined accordingly.

The Schrödinger form with Neumann boundary conditions $ \E^{(N)}_{\nabla,W}$ is defined by
\begin{align*}
D( \E^{(N)}_{\nabla,W})&=\{\Phi\in W^{1,2}(M;E)\mid \langle W\Phi|\Phi\rangle\in L^2(M)\},\\
 \E^{(N)}_{\nabla,W}(\Phi)&=\int_M \abs{\nabla\Phi}^2\,d\mu+\int_M\langle W\Phi|\Phi\rangle\,d\mu.
\end{align*}
\end{definition}
Just as in the scalar case one shows that $ \E^{(N)}_{\nabla,W}$ is
closed.

\begin{definition}[Schrödinger form with Dirichlet boundary conditions]
The Schrödinger form with Dirichlet boundary conditions $\E^{(N)}_{\nabla,W}$ is the closure of the restriction of $
\E^{(N)}_{\nabla,W}$ to $\Gamma_c(M;E)$.
\end{definition}

In the scalar case when $\nabla$ is simply the exterior derivative $d$
on functions, we will write $\E^{(N)}_W$ and $\E^{(D)}_W$ for $
\E^{(N)}_{d,W}$ and $ \E^{(D)}_{d,W}$ respectively, and simply
$\E^{(N)}$ and $\E^{(D)}$ if $W=0$. It is well-known that the forms $\E^{(N)}_W$ and $\E^{(D)}_W$ are Dirichlet forms (see e.g. \cite{Fu}, Section 1.2). In particular, they satisfy the first Beurling-Deny criterion.

We will now establish that $ \E^{(N)}_{\nabla,W}$ is smoothly inner
regular in the sense of Definition \ref{smoothly_inner_reg}. This result should be well-known, but since we could not find a reference, we outline its proof here. We
write $W^{1,2}_c(M;E)$ for $W^{1,2}(M;E)\cap L^2_c(M;E)$, the space
of all sections in $W^{1,2}$ with compact support.

\begin{lemma}\label{covariant_inner_reg}
The space $\Gamma_c(M;E)$ is dense in $W^{1,2}_c(M;E)$ and $D(
\E^{(N)}_{\nabla,W})\cap L^2_c(M;E)$.
\end{lemma}
\begin{proof}
Let $\Omega\subset M$ be open and relatively compact. Relative
compactness ensures that $\Omega$ can be covered by finitely many
charts $(U_1,\phi_1),\dots,(U_n,\phi_n)$ such that $\bigcup_{k=1}^n
U_k$ is relatively compact and such that there exist trivializations
$\psi_k$ for $E|_{U_k}$.

Let $(\lambda_k)$ be a partition of unity subordinate to $(U_k)$.
For $\Phi\in L^2(M;E)$ define $\Phi_k$ by
\begin{align*}
\Phi_k\colon
\phi_k(U_k)\overset{\phi_k^{-1}}{\lra}U_k\overset{\lambda_k\Phi}{\lra}E|_{U_k}\overset{\psi_k^{-1}}{\lra}U_k\times
\IC^n\overset{\mathrm{pr}_2}{\lra}\IC^n.
\end{align*}
In particular, if $\Phi\in W^{1,2}(M;E)$, then $\Phi_k\in
W^{1,2}_c(\phi_k(U_k);\IC^n)$. Moreover, since $\bigcup_k U_k$ is
relatively compact, there are constants $c,C>0$ such that
\begin{align*}
c\norm{\Phi_k}_{W^{1,2}(\phi_k(U_k);\IC^n)}^2\leq
\int_{U_k}(\abs{\lambda_k \Phi}^2+\abs{\nabla (\lambda_k\Phi)}^2)\,d\mu\leq
C\norm{\Phi_k}_{W^{1,2}(\phi_k(U_k);\IC^n)}^2
\end{align*}
for all $\Phi\in W^{1,2}(M;E)$, $k\in \{1,\dots,N\}$, as can be
easily seen from the local coordinate expression for $\nabla$.  It
is well-known that $C_c^\infty(\phi_k(U_k);\IC^n)$ is dense in
$W^{1,2}_c(\phi_k(U_k);\IC^n)$. Using the norm estimate above and
the partition of unity, this gives the density of $\Gamma_c(M;E)$ in
$W^{1,2}_c(M;E)$.

This argument can easily be extended to the case of non-vanishing
potentials, we just wanted to avoid the additional notation.
\end{proof}

In the following lemma we collect some useful calculus rules for the
weak covariant derivative. They can be derived from the
corresponding rules for smooth sections by first noticing that they
are local, hence it suffices to verify them for compactly supported
sections, and then using the density of $\Gamma_c(M;E)$ in
$W^{1,2}_c(M;E)$.
\begin{lemma}
If $v\in W^{1,2}_\loc(M)$, $\Phi,\Psi\in W^{1,2}_\loc(M;E)$ and
$f\colon \IC\lra \IC$ is smooth and Lipschitz, then
$\nabla(v\Phi)\in L^1_\loc(M;E\otimes T^\ast_\IC M)$, $d\langle
\Phi|\Psi\rangle\in L^1_\loc(M;T^\ast_\IC M)$, $f\circ v\in
W^{1,2}_\loc(M)$ and
\begin{align*}
\nabla(v\Phi)&=v\nabla \Phi+\Phi\otimes d v,\\
d\langle \Phi|\Psi\rangle&=\langle \nabla \Phi|\Psi\rangle+\langle\Phi|\nabla\Psi\rangle,\\
d(f\circ v)&=(f'\circ v)dv.
\end{align*}
\end{lemma}

\begin{proposition}\label{domination_manifolds}
If $(E,\nabla,W)$ is a regular Schrödinger bundle and $V\in
L^1_\loc(M)_+$ such that $W(x)\geq V(x)\mathrm{id}_x$ in the sense of quadratic forms for a.e. $x\in M$, then $
\E^{(N)}_{\nabla,W}$ is dominated by $\E^{(N)}_{W}$.
\end{proposition}
\begin{proof}
\emph{Step 1.} If $\Phi\in D( \E^{(N)}_{\nabla,W})$, then
$\abs{\Phi}\in D(\E^{(N)})$:

Let $\abs{\Phi}_\epsilon=(\abs{\Phi}^2+\epsilon^2)^{1/2}$. For
$\Phi\in \Gamma_c(M;E)$ it was shown in \cite{HSU80}, Section 2,
that $\abs{d\abs{\Phi}_\epsilon}\leq \abs{\nabla \Phi}$. Since
$\Gamma_c(\Omega;E)$ is dense in $W^{1,2}_c(\Omega;E)$ for
$\Omega\subset M$ open, relatively compact and the inequality is
local, we conclude that $\abs{d\abs{\Phi}_\epsilon}\leq \abs{\nabla
\Phi}$ for all $\Phi\in W^{1,2}(M;E)$.

Since $\abs{\Phi}_\epsilon\to\abs{\Phi}$ in $L^2_\loc(M)$, we can
use the lower semicontinuity of the energy integral to get
\begin{align*}
\int_{\Omega}\abs{d\abs{\Phi}}^2\,d\mu\leq\liminf_{\epsilon\searrow
0}\int_\Omega \abs{d\abs{\Phi}_\epsilon}^2\,du\leq \int_\Omega
\abs{\nabla \Phi}^2\,d\mu
\end{align*}
for all open, relatively compact $\Omega\subset M$. This inequality
implies $\abs{d\abs{\Phi}}\leq \abs{\nabla \Phi}\in L^2(M)$.

Finally, $V\abs{\Phi}^2\leq \langle W\Phi|\Phi\rangle\in L^2(M)$ by
assumption. Thus $d\abs{\Phi}\in D(\E^{(N)}_{V})$.

\emph{Step 2.} If $v\in D(\E^{(N)})$ and $\Phi\in D(
\E^{(N)}_{\nabla,W})$ with $0\leq v\leq\abs{\Phi}$, then $v\sgn
\Phi\in D( \E^{(N)}_{\nabla,W})$:

Let $\Psi_\epsilon=v\Phi/\abs{\Phi}_\epsilon$. Using
$\abs{\Phi}_\epsilon\geq \epsilon$ and the chain, rule we get
$d\frac
1{\abs{\Phi}_\epsilon}=-d\abs{\Phi_\epsilon}/\abs{\Phi}_\epsilon^2$.
The product rule implies
\begin{align*}
\nabla\frac{\Phi}{\abs{\Phi_\epsilon}}=\frac
1{\abs{\Phi}_\epsilon}\nabla
\Phi-\Phi\otimes\frac{d\abs{\Phi}_\epsilon}{\abs{\Phi}_\epsilon^2}.
\end{align*}
As $\abs{\Phi}_\epsilon\geq \epsilon$, the right-hand side is in
$L^2_\loc(M;E\otimes T^\ast_\IC M)$. Hence we can apply the product
rule a second time to get
\begin{align*}
\nabla\Psi_\epsilon=\frac {\Phi}{\abs{\Phi}_\epsilon}\otimes
dv+\frac{v}{\abs{\Phi}_\epsilon}\nabla \Phi-v\Phi\otimes
\frac{d\abs{\Phi}_\epsilon}{\abs{\Phi}_\epsilon^2}.
\end{align*}
Thus
\begin{align*}
\abs{\nabla\Psi_\epsilon}\leq
\abs{dv}\frac{\abs{\Phi}}{\abs{\Phi}_\epsilon}+\frac{v}{\abs{\Phi}}(\abs{\nabla\Phi}+\abs{d\abs{\Phi}_\epsilon})\leq
\abs{dv}+2\abs{\nabla\Phi}.
\end{align*}
Moreover,
\begin{align*}
\langle
W\Psi_\epsilon|\Psi_\epsilon\rangle=\frac{v^2}{\abs{\Phi}_\epsilon^2}\langle
W\Phi|\Phi\rangle\leq \langle W\Phi|\Phi\rangle.
\end{align*}
As in Step 1, a limiting argument gives $v\sgn \Phi\in D(
\E^{(N)}_{\nabla,W})$.

\emph{Step 3.} If $v,\Phi$ as in Step 2, then $\Re
\E^{(N)}_{\nabla,W}(\Phi,v\sgn \Phi)\geq
\E^{(N)}_{V}(\abs{\Phi},v)$:

Let $\Psi_\epsilon=v \Phi/\abs{\Phi}_\epsilon$ as in Step 2. By what
we have already established in Steps 1 and 2,
\begin{align*}
\Re\langle \nabla \Phi|\nabla \Psi_\epsilon\rangle&=\Re\left\langle \nabla \Phi\bigg|\frac{\Phi}{\abs{\Phi}_\epsilon}\otimes dv+\frac v{\abs{\Phi}_\epsilon}\nabla \Phi-v\Phi\otimes \frac{d\abs{\Phi}_\epsilon}{\abs{\Phi}_\epsilon^2}\right\rangle\\
&=\frac 1{\abs{\Phi}_\epsilon}\langle\Re \langle \nabla \Phi|\Phi\rangle|dv\rangle+\frac{v}{\abs{\Phi}_\epsilon}(\abs{\nabla \Phi}^2-\frac 1{\abs{\Phi}_\epsilon}\langle \Re \nabla \Phi|\Phi\rangle| d\abs{\Phi}_\epsilon\rangle)\\
&=\langle d\abs{\Phi}_\epsilon|dv\rangle+\frac {v}{\abs{\Phi}_\epsilon}(\abs{\nabla\Phi}^2-\abs{d\abs{\Phi}_\epsilon}^2)\\
&\geq \langle d\abs{\Phi}_\epsilon|dv\rangle.
\end{align*}
Furthermore,
\begin{align*}
\langle W\Phi|\Psi_\epsilon\rangle=\frac
{v}{\abs{\Phi}_\epsilon}\langle W\Phi|\Phi\rangle\geq \frac
{\abs{\Phi}}{\abs{\Phi}_\epsilon} V\abs{\Phi}v.
\end{align*}
Letting $\epsilon\to 0$ we obtain the desired inequality.
\end{proof}

\begin{remark}\label{kato-formal}
 a) A distributional  version
of Kato's inequality in this setting was first proven  by Hess,
Schrader, Uhlenbrock \cite{HSU80} (for compact manifolds and
vanishing potentials) based on arguments originally due to Kato
\cite{Kat72} for magnetic Schr\"odinger operators. Their
considerations do not  include discussion of domains of the
operators or forms and, therefore,
 do
not allow one to conclude domination. Our reasoning, which relies on
the same method, can be seen as a completion of their result.

b) For open manifolds, the same domination has been proven by
G\"uneysu (see \cite{Gue14}, proof of Proposition 2.2) for
\emph{Dirichlet boundary conditions} (and vanishing potentials)
using methods from stochastic analysis and the semigroup
characterization of domination. Of course,  both our result and the
result of \cite{Gue14} apply to  those  situations where Dirichlet-
and Neumann boundary conditions agree.
\end{remark}

Lemma \ref {covariant_inner_reg} and Proposition
\ref{domination_manifolds} ensure that the conditions of Corollary
\ref{uniqueness_smooth} are met, so that we obtain the following
main result of this section.

\begin{theorem}\label{uniqueness_Schrödinger_manifolds}
Let $(E,\nabla,W)$ be a regular Schrödinger bundle and $V\in
L^1_\loc(M)_+$ such that $W(x)\geq V(x)\mathrm{id}_x$ in the sense of quadratic
forms for a.e. $x\in M$. If $\E^{(N)}_V=\E^{(D)}_V$, then $
\E^{(N)}_{\nabla,W}= \E^{(D)}_{\nabla,W}$.
\end{theorem}

Finally, we discuss how Theorem
\ref{uniqueness_Schrödinger_manifolds} enables us to apply capacity
estimates to form uniqueness problems for Schrödinger operators on
bundles. The scalar case has been treated in \cite{GM13}; we follow
their terminology.

\begin{definition}[Cauchy boundary, capacity]
Denote by $\tilde M$ the metric completion of $M$. The Cauchy boundary $\partial_C M$ of $M$ is defined as $\partial_C M=\tilde M\setminus M$. The capacity of an open subset $\Omega$ of $\tilde
M$ is defined as
\begin{align*}
\capty(\Omega)=\inf\{\norm{u}_{W^{1,2}(M)}^2\mid u\in
W^{1,2}(M),\,0\leq u\leq 1,\,u|_{\Omega\cap M}=1\}.
\end{align*}
As usual, the infimum of the empty set is taken to be $\infty$. The
capacity is extended to arbitrary subsets $\Sigma$ of $\Omega$
by setting
\begin{align*}
\capty(\Sigma)=\inf_{\Omega\supset\Sigma\text{ open}}\capty(\Omega).
\end{align*}
A subset $\Sigma$ of $\tilde M$ is called \emph{polar} if
$\capty(\Sigma)=0$.
\end{definition}

\begin{corollary}\label{capacity_estimates}
Consider the following assertions:
\begin{itemize}
\item[(i)] The Cauchy boundary $\partial_C M$ of $M$ is polar.
\item[(ii)] $\E^{(N)}=\E^{(D)}$.
\item[(iii)] $ \E^{(N)}_{\nabla,W}= \E^{(D)}_{\nabla,W}$ for all regular Schrödinger bundles $(E,\nabla,W)$.
\end{itemize}
Then (i)$\implies$(ii)$\iff$(iii). Moreover, if there exists an
exhaustion $(B_k)$ of $\tilde M$ such that $\capty (B_k\setminus
M)<\infty$ for all $k\in\IN$, then all three assertions are
equivalent.
\end{corollary}
\begin{proof}
The implication (ii)$\implies$(iii) is content of the last theorem,
while (iii)$\implies$(ii) is trivial. The remaining statements
follow from \cite{GM13}, Theorem 1.7 and Lemma 2.2.
\end{proof}

\begin{remark}
It is an achievement of our abstract main theorem to make concepts
from potential theory such as the capacity in this corollary
directly applicable to Schr\"odinger forms, which are not Dirichlet
forms, without having to go through the uniqueness proof for the
scalar case again.
\end{remark}

\section{Applications to magnetic Schr\"odinger forms on graphs}\label{Applications-grph}
In this section we will study discrete analoga of the Laplacian
respectively magnetic Schr\"odinger operators in Euclidean space.
Analysis on graphs has been an active field of research in recent
years and uniqueness of extensions of operators respectively forms
on graphs have been intensively studied. We just point to
\cite{HKLW, HKMW,KL12,Woj08,Woj09} for non-magnetic forms and \cite{GKS16,MT14,MT15} for magnetic forms as a
few examples and refer to the recent survey \cite{Schm20} for further reference. The latter also discusses results from a preprint version of the present paper.

Compared to the Euclidean case, the discrete setting allows more
clarity in the presentation as some mere technical complications do
not appear. In particular, Corollary
\ref{corollary_uniqueness_forms} can be applied directly since
$L_c^\infty(X)$ and $C_c(X)$ coincide for discrete spaces.

We will start with some basic definitions, including those of
magnetic Schr\"odinger forms on graphs (Definitions
\ref{def_magnetic_graph_Neumann} and
\ref{def_magnetic_graph_Dirichlet}), essentially following
\cite{KL10},\cite{KL12} for graphs and Dirichlet forms over
discrete spaces and \cite{MT15} for vector bundles over graphs
and magnetic Schr\"odinger operators. Then we show that the form
with magnetic field is dominated by the form without magnetic field
(Proposition \ref{domination_magnetic_forms}) before we finally give
the uniqueness result (Theorem \ref{form_uniqueness_graphs}) and
discuss some examples.

\begin{definition}[Weighted graph]
A weighted graph $(X,b,m)$ consists of an (at most) countable set
$X$, an edge weight $b\colon X\times X\lra [0,\infty)$ and a measure $m\colon
X\lra(0,\infty)$, where $b$ satisfies
\begin{itemize}
\item[(b1)]$b(x,x)=0$ for all $x\in X$,
\item[(b2)]$b(x,y)=b(y,x)$ for all $x,y\in X$,
\item[(b3)]$\sum_{y\in X}b(x,y)<\infty$ for all $x\in X$.
\end{itemize}
\end{definition}
Observe that we do not assume our graphs to be locally finite, that
is, $\{y\in X\mid b(x,y)>0\}$ may be infinite as long the edge
weights are still summable (at each vertex). We shall regard $X$ as
a discrete topological space. Consequently, $C_c(X)$ is the space of
functions on $X$ with finite support.  We regard $m$ as a measure on
the power set  $\mathcal{P}(X)$ of $X$ via
\begin{align*}
m(A):=\sum_{x\in A} m(x),\;A\subset X,
\end{align*}
and denote the corresponding $L^2$-space by $\ell^2(X,m)$.  

Any weighted graph $(X,b,m)$ and a function $c\colon X\lra[0,\infty)$  come  with a formal Laplacian $\tilde L$ acting on
$$
\{f : X\longrightarrow \IR : \sum_{y} b(x,y) |f(y)| <\infty
\mbox{ for all $x\in X$}\}$$
by
$$\tilde L f (x) = \frac{1}{m(x)} \sum_{y} b(x,y) (f(x) - f(y))
+ \frac{c(x)}{m(x)} f(x)$$ for $x\in X$.

For the next definition recall that a Hermitian vector bundle over a discrete base space $X$ is just a family of isometrically isomorphic finite-dimensional Hilbert spaces indexed by $X$.
\begin{definition}[Unitary connection]
Let $X$ be a discrete space and $E\lra X$ a Hermitian vector bundle. A connection on $E$ is a family $\Phi=(\Phi_{x,y})_{x,y\in X}$ of unitary maps $\Phi_{x,y}\colon E_y\lra E_x$ such that $\Phi_{x,y}=\Phi_{y,x}^{-1}$.
\end{definition}
For a Hermitian vector bundle $E$ over a discrete space $X$ equipped with a measure $m$ we denote by
\begin{align*}
\Gamma(X;E)&=\{u\colon X\lra\prod_{x\in X}E_x\mid u(x)\in E_x\},\\
\Gamma_c(X;E)&=\{u\in\Gamma(X;E)\mid \supp u\text{ finite}\},\\
\ell^2(X,m;E)&=\{u\in\Gamma(X;E)\mid\sum_{x\in X}\langle u(x),u(x)\rangle_x m(x)<\infty\}
\end{align*}
the space of all sections, the space of all sections with compact support and the space of all $L^2$-sections. The latter becomes a Hilbert space equipped with the inner product
\begin{align*}
\langle \cdot,\cdot\rangle_{\ell^2(X,m;E)}\colon\ell^2(X,m;E)\times\ell^2(X,m;E)\lra\IC,\,(u,v)\mapsto\sum_{x\in X}\langle u(x),v(x)\rangle_x m(x).
\end{align*}
A bundle endomorphism $W$ of a Hermitian vector bundle $E$ over a discrete base space $X$ is a family $(W(x))_{x\in X}$ of linear maps $W(x)\colon E_x\lra E_x$.

For the remainder of the section, $(X,b,m)$ is a weighted graph,   $c\colon X\lra[0,\infty)$, 
$E$ a Hermitian vector bundle over $X$ with unitary connection
$\Phi$ and $W$ a bundle endomorphism of $E$ that is pointwise
positive, that is, $\langle W(x)v,v\rangle_x\geq 0$ for all $x\in
X,\,v\in E_x$.

\smallskip

Now we can define the basic object of our interest, the magnetic
Schr\"odinger form (with Dirichlet and Neumann boundary conditions).

\begin{definition}[Magnetic form with Neumann boundary conditions]\label{def_magnetic_graph_Neumann}
For $u\in \Gamma(X;E)$ let
\begin{align*}
\tilde Q_{\Phi,b,W}(u)=\frac 1 2\sum_{x,y}b(x,y)\abs{u(x)-\Phi_{x,y}u(y)}_x^2+\sum_{x}\langle W(x)u(x),u(x)\rangle_x\in[0,\infty].
\end{align*}
The magnetic Schr\"odinger form with Neumann boundary conditions is
defined via
\begin{align*}
D(Q^{(N)}_{\Phi,b,W})&=\{u\in\ell^2(X,m;F)\mid\tilde Q_{\Phi,b,W}(u)<\infty\},\\
Q^{(N)}_{\Phi,b,W}(u)&=\tilde Q_{\Phi,b,W}(u).
\end{align*}
\end{definition}

To shorten notation, we write $\norm{\cdot}_{\Phi,b,W}$ for the
form norm of $Q^{(N)}_{\Phi,b,W}$. By the same arguments as in the
Dirichlet form case, the form $Q^{(N)}_{\Phi,b,W}$ is closed (see
\cite{KL12}, Lemma 2.3).

\begin{definition}[Magnetic form with Dirichlet boundary conditions]\label{def_magnetic_graph_Dirichlet}
The magnetic Schr\"odinger form with Dirichlet boundary conditions
$Q^{(D)}_{\Phi,b,W}$ is the closure of the restriction of
$Q^{(N)}_{\Phi,b,W}$ to $C_c(X)$.
\end{definition}

If $E_x=\IC$ endowed with the standard inner product, the action of $W$ is multiplication by a function $c \colon X \lra [0,\infty)$. If, furthermore,  $\Phi_{x,y}=1$ for all $x,y\in X$ we will suppress $\Phi$ in the index and simply write $Q^{(D)}_{b,c}$ (resp. $Q^{(N)}_{b,c})$. We
may also drop other indices if they are clear from the context. The
interest in these forms is particularly a result of the fact that
$Q^{(D)}_{b,c}$ and $Q^{(N)}_{b,c}$ are Dirichlet forms. Indeed, all
regular Dirichlet forms over a discrete measure space are of the
form $Q^{(D)}_{b,c}$ for some graph $(X,b,m)$ and function  $c \colon X \lra [0,\infty)$ (cf. \cite{KL12},
Lemma 2.2). This is one motivation to study also graphs that are not
locally finite.

For our subsequent considerations we also note that the generators
of  both $Q^{(D)}_{b,c}$ and $Q^{(N)}_{b,c}$ are restrictions of
$\tilde L$ (see \cite{HKLW}). So, if $\tilde{L}$  has only one self-adjoint realization, then $Q^{(D)}_{b,c} = Q^{(N)}_{b,c}$ follows. In particular this is the case if the restriction
$$L_0 := \tilde L |_{C_c (X)}$$
of $\tilde L$ to $C_c (X)$ maps into $\ell^2 (X,m)$ and is essentially self-adjoint.

\smallskip

As a next step to establish criteria for $Q^{(N)}_\Phi=Q^{(D)}_\Phi$
we will show that the form with magnetic field is dominated by the
non-magnetic form. First we prove an easy technical lemma.

\begin{lemma}\label{inequality_sgn}
Let $V$ be a Hilbert space, $a,b\in V$, and $\alpha,\beta\geq 0$ with $\alpha\leq\norm{a},\,\beta\leq\norm{b}$. Define
\begin{align*}
\tilde a=\begin{cases}\frac{\alpha}{\norm a}a&\colon a\neq 0\\0&\colon a=0\end{cases}
\end{align*}
and likewise $\tilde b$. Then
\begin{align*}
\norm{\tilde a-\tilde b}^2\leq\abs{\alpha-\beta}^2+\norm{a-b}^2.
\end{align*}
\end{lemma}
\begin{proof}
If $a=0$ or $b=0$, the inequality is obvious. Hence assume that $a,b\neq 0$.\\
In the following computation we use the inequality $2\lambda\mu\leq\lambda^2+\mu^2$ for $\lambda,\mu\in\IR$.
\begin{align*}
\norm{\tilde a-\tilde b}^2&=\norm{\tilde a}^2+\norm{\tilde b}^2-2\Re\langle\tilde a,\tilde b\rangle\\
&=\alpha^2+\beta ^2-2\Re\langle \tilde a,\tilde b\rangle\\
&=\abs{\alpha-\beta}^2+2\alpha\beta-2\Re \langle\tilde a,\tilde b\rangle\\
&=\abs{\alpha-\beta}^2+2\frac{\alpha\beta}{\norm a\norm b} (\norm a\norm b-\Re\langle a,b\rangle)\\
&\leq \abs{\alpha-\beta}^2+ 2\norm{a}\norm{b}-2\Re\langle a,b\rangle\\
&\leq \abs{\alpha-\beta}^2+\norm{a}^2+\norm{b}^2-2\Re\langle a,b\rangle\\
&=\abs{\alpha-\beta}^2+\norm{a-b}^2\qedhere
\end{align*}
\end{proof}

We will now prove that the magnetic form is dominated by the form
without magnetic field. We note that a pointwise version of Kato's inequality
was already given in \cite{MT15}, Lemma 3.3. This work does not
involve a discussion of validity of this inequality on the domains
of the operators. Hence, it can not be used to conclude domination.
In this sense our result can be seen as a completion of their result
(compare the remark below Proposition \ref{domination_manifolds} for a discussion of a similar
issue as well).

\begin{proposition}\label{domination_magnetic_forms}
Assume that $W(x)\geq c(x)\mathrm{id}_x$ in the sense of quadratic forms for all $x\in X$. Then $Q^{(N)}_{\Phi,b,W}$ is dominated by $Q^{(N)}_{b,c}$.
\end{proposition}
\begin{proof}
We have to show that
$D(Q^{(N)}_{\Phi,b,W})$ is a generalized ideal in $D(Q^{(N)}_{b,c})$
and that
\begin{align*}
\Re Q^{(N)}_{\Phi,b,W}(u,\tilde u)\geq Q^{(N)}_{b,c}(\abs u,\abs {\tilde u})
\end{align*}
holds for all $u,\tilde u\in D(Q^{(N)}_{\Phi,b,W})$ satisfying $\langle u(x),\tilde u(x)\rangle_x=\abs{u(x)}\abs{\tilde u(x)}$ for all $x\in X$.\\
First, let $u\in D(Q^{(N)}_{\Phi,b,W})$. Then $\abs u\in\ell^2(X,m)$ and
\begin{align*}
\tilde Q_{\Phi,b,W}(u)&=\frac 1 2\sum_{x,y}b(x,y)\abs{u(x)-\Phi_{x,y}u(y)}^2+\sum_x\langle W(x)u(x),u(x)\rangle\\
&\geq \frac 1 2\sum_{x,y}b(x,y)\abs{\abs{u(x)}-\abs{u(y)}}^2+\sum_x c(x)\abs{u(x)}^2\\
&=\tilde Q_{b,c}(\abs u),
\end{align*}
hence $\abs u\in D(Q^{(N)}_{b,c})$.\\
Next let $v\in D(Q^{(N)}_{b,c})$ with $0\leq v\leq\abs u$. Obviously, $\norm{v\sgn u}_{\ell^2}\leq\norm v_{\ell^2}$, thus $v\sgn u\in\ell^2(X,m;F)$.\\
Applying Lemma \ref{inequality_sgn} to $V=E_x,a=u(x),b=\Phi_{x,y}u(y),\alpha=v(x),\beta=v(y)$, we obtain
\begin{align*}
\abs{v(x)\sgn u(x)-\Phi_{x,y}v(y)\sgn u(y)}^2\leq\abs{v(x)-v(y)}^2+\abs{u(x)-\Phi_{x,y}u(y)}^2.
\end{align*}
Summation over $x,y$ implies
\begin{align*}
\tilde Q_{\Phi,b,0}(v\sgn u)\leq  Q^{(N)}_{b,0}(v)+Q^{(N)}_{\Phi,b,0}(u).
\end{align*}
Furthermore,
\begin{align*}
\sum_x\langle W(x)v(x)\sgn u(x),v(x)\sgn u(x)\rangle&\leq\sum_x \abs{u(x)}^2\langle{W(x)\sgn u(x),\sgn u(x)}\rangle\\
&=\sum_x \langle W(x)u(x),u(x)\rangle,
\end{align*}
hence
\begin{align*}
\tilde Q_{\Phi,b,W}(v\sgn u)\leq Q^{(N)}_{b,0}(v)+Q^{(N)}_{\Phi,b,0}(u)+\sum_x c(x)\abs{u(x)}^2\leq Q^{(N)}_{b,c}(v)+Q^{(N)}_{\Phi,b,W}(u),
\end{align*}
that is, $v\sgn u\in D(Q^{(N)}_{\Phi,b,W})$.\\
Let $u,\tilde u\in D(Q^{(N)}_{\Phi,b,W})$ such that $\langle u(x),\tilde u(x)\rangle_x=\abs{u(x)}\abs{\tilde u(x)}$ for all $x\in X$. Then we have
\begin{align*}
&\Re\langle u(x)-\Phi_{x,y}u(y),\tilde u(x)-\Phi_{x,y}\tilde u(x)\rangle\\
={}&\Re(\langle u(x),\tilde u(x)\rangle-\langle u(x),\Phi_{x,y}\tilde u(y)\rangle-\langle\Phi_{x,y} u(y),\tilde u(x)\rangle+\langle u(y),\tilde u(y)\rangle)\\
={}&\abs{u(x)}\abs{\tilde u(x)}+\abs{u(y)}\abs{\tilde u(y)}-\Re\langle u(x),\Phi_{x,y}\tilde u(y)\rangle-\Re\langle\Phi_{x,y} u(y),\tilde u(x)\rangle\\
\geq{}& \abs{u(x)}\abs{\tilde u(x)}+\abs{u(y)}\abs{\tilde u(y)}-\abs{u(x)}\abs{\tilde u(y)}+\abs{u(y)}\abs{\tilde u(x)}\\
={}&(\abs{u(x)}-\abs{u(y)})(\abs{\tilde u(x)}-\abs{\tilde u(y)}).
\end{align*}
After multiplication with $b(x,y)$ and summation over $x,y\in X$ we get
\begin{align*}
\Re Q^{(N)}_{\Phi,b, W}(u,\tilde u)\geq Q^{(N)}_{b,c}(\abs u,\abs{\tilde u}).&\qedhere
\end{align*}
\end{proof}

\begin{corollary}
The form $Q_{b,c}^{(N)}$
is dominated by $Q^{(N)}_{b,0}$.
\end{corollary}
\begin{proof}
This follows from Proposition \ref{domination_magnetic_forms} by
taking $E_x=\IC$, $W(x)=c(x)$ and $\Phi_{x,y}=1$ for all $x,y\in X$.
\end{proof}

Having proven the domination property, the announced main result of
this section is now an easy consequence of Corollary
\ref{corollary_uniqueness_forms}. In a very informal way it says
that adding a magnetic and electric field does not disturb the form
uniqueness.

\begin{theorem}\label{form_uniqueness_graphs}
Assume that $W(x)\geq c(x)\mathrm{id}_x$ in the sense of quadratic forms for
all $x\in X$. If $Q^{(D)}_{b,c}=Q^{(N)}_{b,c}$, then
$Q^{(D)}_{\Phi,b,W}=Q^{(N)}_{\Phi,b,W}$.
\end{theorem}
\begin{proof}
We have proven in Proposition \ref{domination_magnetic_forms} that
$Q^{(N)}_{\Phi,b,W}$ is dominated by $Q^{(N)}_{b,c}$. An application
of Corollary \ref{corollary_uniqueness_forms} for $\a=Q^{(N)}_{\Phi,b,W}$ and $\b=Q^{(N)}_{b,c}$ yields the
claim.
\end{proof}

To apply the theorem, we need $Q^{(D)}_{b,c}=Q^{(N)}_{b,c}$.  There
are quite a few conditions under which this equality  holds. The
first were phrased in terms of the measure $m$ and the combinatorial
graph structure:

\begin{example}[See \cite{KL12} for details and proofs]
If $\tilde L C_c(X)\subset \ell^2(X,m)$ and $\sum_{n=1}^\infty
m(x_n)=\infty$ for any sequence $(x_n)$ in $X$ such that
$b(x_n,x_{n+1})>0$ for all $n\in \IN$, then $L_0$ is essentially
self-adjoint and all form extensions of $Q^{(D)}$ coincide with
$Q^{(D)}$. The given conditions are in particular satisfied if
$\inf_{x\in X}m(x)>0$.
\end{example}

It turns out that the concept of intrinsic pseudo metrics
(introduced for general not necessarily local Dirichlet forms in
\cite{FLW}) provides a suitable framework for conditions for
uniqueness of form extensions. Here, a pseudo metric $d\colon
X\times X\lra [0,\infty)$ is called intrinsic for $(X,b,m)$ if
\begin{align*}
\frac 1{m(x)}\sum_ y b(x,y) d(x,y)^2\leq 1
\end{align*}
for all $x\in X$.

A pseudo metric $d$  is called a path pseudo
metric if there is a function $\sigma\colon X\times
X\lra[0,\infty)$, satisfying $\sigma(x,y)=\sigma(y,x)$ and
$\sigma(x,y)>0$ iff $b(x,y)>0$ for all $x,y\in X$, such that
\begin{align*}
d(x,y)=d_\sigma(x,y):=\inf_{\gamma}\sum_{k=1}^n \sigma(x_{k-1},x_k)
\end{align*}
where the infimum is taken over all paths $(x_0,\ldots,x_n)$
connecting $x$ and $y$.

The results discussed in the next two examples are taken from \cite{GKS16}, Theorem~2.16. Under slightly stronger assumptions they were first established in \cite{HKMW}, Corollary~1 and Theorem~2. Further examples can be found there and in the recent survey \cite{Schm20}.

\begin{example}\label{example:finite distance balls} Let $d$ be an intrinsic pseudo metric on $(X,b,m)$ and $c \colon X \lra [0,\infty)$. If balls with respect to $d$ are finite, then $\tilde L$ has a unique self-adjoint realization, and consequently $Q_{b,c}^{(D)}=Q_{b,c}^{(N)}$. 
\end{example}

\begin{remark} The previous example shows the strength of our method as form uniqueness on bundles under the condition of finite intrinsic balls has not been treated in earlier works. Notice that in general we have only $Q^{(D)}=Q^{(N)}$ and not the stronger assertion that $L_0$ is essentially self-adjoint.
\end{remark}

On locally finite graphs path metrics satisfy a discrete Hopf-Rinow theorem (see \cite{HKMW}, Theorem~A.1) and the formal Laplacian $\tilde L$ maps $C_c(X)$ to $\ell^2(X,m)$.  In this case the previous example can therefore be restated as follows (see \cite{GKS16}, Theorem~2.16 for details). 

\begin{example}
If $(X,b,m)$ is locally finite and there is an intrinsic path 
metric $d$ such that $(X,d)$ is metrically complete, then $L_0$ is
essentially self-adjoint, and consequently $Q_{b,c}^{(D)}=Q_{b,c}^{(N)}$.
\end{example}


There is also a non-geometric version for the condition in these examples. In some recent works form uniqueness respectively essential self-adjointness is inferred from  two similar assumptions called {\em completeness} and {\em $\chi$-completeness}. In \cite{HL15} a graph $(X,b,m)$ is called \emph{complete} if there is a non-decreasing sequence $(\eta_k)$ in $C_c(X)$ such that $\eta_k\to 1$ pointwise and  the norm of the discrete gradient satisfies
\begin{align*}
\frac 1{m(x)}\sum_y b(x,y)\abs{\eta_k(x)-\eta_k(y)}^2 \to 0, \quad k \to \infty,
\end{align*}
uniformly in $x \in X$. In \cite{AT15} for locally finite graphs $\chi$-completeness is introduced as above with the uniform convergence of the norm of the discrete gradient to $0$ replaced by uniform boundedness.  These concepts were coined after a preprint of \cite{FLW} was published, but it is one of the main technical insights thereof that both are implied by the existence of an intrinsic metric with finite balls, see \cite{FLW} Proposition~4.7 in the context of regular Dirichlet forms or \cite{HL15} Theorem~2.1 for a discrete version.

The discussion in  \cite{AT15,HL15} suggest that ($\chi$-) completeness is a less restrictive condition than the existence of an intrinsic metric with finite distance balls, but this is not the case. To fill this gap in the literature we give a proof for the equivalence of completeness and the existence of an intrinsic metric with finite distance balls in Appendix~\ref{appendix}. Moreover, we show that both are equivalent to $\chi$-completeness provided the graph is locally finite. In summary,  completeness and $\chi$-completeness are point free versions of the condition in Example~\ref{example:finite distance balls} that ensures $Q_{b,c}^{(D)}=Q_{b,c}^{(N)}$.

\appendix

\section{Completeness of graphs and intrinsic metrics} \label{appendix}

For the definitions used in this Appendix we refer to Section~\ref{Applications-grph}. Given a graph $(X,b,m)$ and $f \colon X \to \IC$ we write
$$\Gamma(f)(x) = \frac{1}{m(x)} \sum_{y} b(x,y) |f(x) - f(y)|^2. $$
Moreover, by $C_0(X)$ we denote the closure  of $C_c(X)$ in $\ell^\infty(X)$.
 
 \begin{theorem}
For a graph $(X,b,m)$ the following assertions are equivalent.
\begin{itemize}
 \item[(i)] $(X,b,m)$ admits an intrinsic pseudo metric with finite balls. 
 \item[(ii)] $(X,b,m)$ is complete, i.e., there exists a non-decreasing sequence $(\eta_k)$ in $C_c(X)$ such that $\eta_k\to 1$ pointwise and $\Gamma(\eta_k)\to 0$ uniformly as $k\to\infty$.
\end{itemize}
 If, additionally, for every $y \in X$ the function $X \to \IR, \, x \mapsto b(x,y)/m(x)$ belongs to $C_0(X)$, then these assertions are equivalent to the following.
 \begin{itemize}
  \item[(iii)]  $(X,b,m)$ is $\chi$-complete, i.e., there exists a non-decreasing sequence $(\eta_k)$ in $C_c(X)$ such that $\eta_k\to 1$ pointwise and $(\Gamma(\eta_k))$ is uniformly bounded.
 \end{itemize}
\end{theorem}
\begin{proof}
(i) $\Longrightarrow$ (ii)/(iii): If $d$ is an intrinsic pseudo metric with finite balls, then we can take $\eta_k=(1-d(\cdot,B_k(o))/k)_+$, see e.g. \cite{HL15}, Theorem~2.8.

(ii) $\Longrightarrow$ (i): Assume that $(X,b,m)$ is complete and let $(\eta_k)$ be a sequence as in the definition of completeness. Since $\eta_k\to 1$ pointwise and $X$ is at most countable, we can assume w.l.o.g. that
\begin{align*}
\sum_{k=1}^\infty (1-\eta_k(x))^2<\infty
\end{align*}
for all $x\in X$. 
Otherwise taking a subsequence, we can additionally assume
\begin{align*}
\sum_{k=1}^\infty \norm{\Gamma(\eta_k)}_\infty\leq 1.
\end{align*}
For $x,y\in X$ let
\begin{align*}
d(x,y)=\left(\sum_{k=1}^\infty (\eta_k(x)-\eta_k(y))^2\right)^{1/2}.
\end{align*}
First note that $d(x,y)$ is finite:
\begin{align*}
d(x,y)\leq \left(\sum_{k=1}^\infty(1-\eta_k(x))^2\right)^{1/2}+\left(\sum_{k=1}^\infty(1-\eta_k(y))^2\right)^{1/2}<\infty.
\end{align*}
It is then clear that $d$ is a pseudo metric.

It is moreover intrinsic:
\begin{align*}
\frac  1{m(x)}\sum_y b(x,y) d(x,y)^2&=\frac 1{m(x)}\sum_y b(x,y)\sum_{k=1}^\infty (\eta_k(x)-\eta_k(y))^2\\
&=\sum_{k=1}^\infty\Gamma(\eta_k)(x)\\
&\leq 1.
\end{align*}
Finally, $d$ has finite balls. Indeed, let $x\in X$ and $i\in\IN$ with $\eta_i(x)\geq \frac 1 2$. If $j\geq i$ and $y\notin \supp \eta_{j+1}$,  we obtain using the monotonicity of the sequence $(\eta_k)$
\begin{align*}
d(x,y)^2\geq \sum_{k=i}^{j+1} (\eta_k(x)-\eta_k(y))^2\geq \frac{j-i}4.
\end{align*}
Hence if $j>4R^2+i$, then $B_R^d(x)\subset \supp \eta_{j+1}$.

Now, additionally assume that for every $y \in X$ the function $X \to \IR, \, x \mapsto b(x,y)/m(x)$ belongs to $C_0(X)$.

(iii) $\Longrightarrow$ (ii): Let $(\eta_k)$ be a sequence as in the definition of $\chi$-completeness. Since $\Gamma(|\eta_k|) \leq \Gamma(\eta_k)$, we can assume that $\eta_k$ is nonnegative. The pointwise convergence $\eta_k \to 1$,  $0 \leq \eta_k  \leq 1$ and the summability condition (b3) imply  $ \Gamma(\eta_k) \to 0$ pointwise as $ k \to \infty$. Moreover, for $x \in X \setminus {\rm supp}\,\eta_k$ we have
$$ \Gamma(\eta_k)(x) \leq \frac{1}{m(x)} \sum_{y \in {\rm supp}\, \eta_k} b(x,y).$$
Since ${\rm supp}\, \eta_k$ is finite, our additional assumption yields $ \Gamma(\eta_k) \in C_0(X)$. The dual of the Banach space $(C_0(X),\lVert\cdot\rVert_\infty)$ is $\ell^1(X)$. Hence, by Lebesgue's theorem the  pointwise convergence and the uniform boundedness of $(\Gamma(\eta_k))$ yield $ \Gamma(\eta_k)^{1/2} \to 0$ weakly in  $(C_0(X),\lVert\cdot\rVert_\infty)$. According to Mazur's theorem, for every $n \in \IN$ the function $0$ belongs to the strong closure of the  convex hull of $\{ \Gamma(\eta_k)^{1/2} \mid k \geq n\}.$ Hence, inductively we may choose indices $N_1 < N_2 < \ldots$ and coefficients $\lambda_{k,n}, k = N_n + 1, \ldots, N_{n+1}, n \in \IN$, such that the convex combinations
$$\sum_{k = N_n + 1}^{N_{n+1}} \lambda_{k,n}  \Gamma(\eta_k)^{1/2} $$
converge to $0$ strongly in $(C_0(X),\lVert\cdot\rVert_\infty)$, as $n \to \infty$. Then
$$\eta'_n = \sum_{k = N_n + 1}^{N_{n+1}} \lambda_{k,n} \eta_k $$
is finitely supported and converges to $1$ pointwise. Moreover, we chose the convex combinations so that  $(\eta_k)$ being non-decreasing implies that also $(\eta'_n)$ is non-decreasing. The triangle inequality yields
$$\Gamma(\eta'_n)^{1/2}\leq \sum_{k = N_n + 1}^{N_{n+1}} \lambda_{k,n} \Gamma(\eta_k)^{1/2}  \to 0$$
uniformly, as $n \to \infty$. Hence, $(\eta_n')$ is a sequence as in the definition of completeness.
\end{proof}

\begin{remark}
 The condition that for every $y \in X$ the function $X \to \IR, \, x \mapsto b(x,y)/m(x)$ belongs to $C_0(X)$ is trivially satisfied for locally finite graphs, as in this case it even has finite support. This is the situation in which $\chi$-completeness was introduced in \cite{AT15}.  Moreover, by the summability condition (b3) that we impose on $b$ it is satisfied if $m$ is bounded from below by a positive constant.
\end{remark}

\begin{remark}
 For  Riemannian manifolds it is well-known that geodesic completeness is equivalent to the existence of a sequence of compactly supported functions $(\eta_k)$ with $\|\nabla \eta_k\|_\infty \to 0$ and $\eta_k \to 1$ pointwise, see e.g. \cite{BGL} Proposition~C.4.1. Since $\Gamma(f)$  is a discrete version of $|\nabla f|^2$,  the existence of an intrinsic metric with finite distance balls respectively completeness can be seen as an abstract version of geodesic completeness. Indeed, a variant of the above proof also works on manifolds.
\end{remark}

\end{document}